\documentclass[a4paper,12pt]{article}

\usepackage{comment}
\usepackage[utf8]{inputenc}
\usepackage[T1]{fontenc}
\usepackage{bbold}
\usepackage{amssymb}
\usepackage{amsmath, amsthm}
\usepackage{amsfonts}
\usepackage{graphics}
\usepackage{bm}
\usepackage{color}
\usepackage{xcolor}
\usepackage{graphicx}
\usepackage{enumerate}
\usepackage[english]{babel}
\usepackage{enumitem}

\newtheorem{coro}{Corollary}
\newtheorem{thm}{Theorem}

\newtheorem{cl}{Claim}
\newtheorem{Lemma}{Lemma}

 \title{On arborescence packing augmentation in hypergraphs}
\author{Pierre Hoppenot, Zolt\'an Szigeti\\ University Grenoble Alpes, Grenoble INP, CNRS,\\ Laboratory  G-SCOP}
\begin{document}
\maketitle

\begin{abstract}
We deepen the link between two classic areas of combinatorial optimization: augmentation and packing arborescences. We consider the following type of questions: What is the minimum number of arcs to be added to a digraph so that in the resulting digraph there exists some special kind of packing of arborescences? We answer this question for  two problems: $h$-regular {\sf M}-independent-rooted $(f,g)$-bounded $(\alpha, \beta)$-limited packing of mixed hyperarborescences and $h$-regular $(\ell, \ell')$-bordered $(\alpha, \beta)$-limited packing of $k$ hyperbranchings. We also solve the undirected counterpart of the latter, that is the augmentation problem for $h$-regular $(\ell, \ell')$-bordered $(\alpha, \beta)$-limited packing of $k$ rooted hyperforests. Our results provide a common generalization of a great number of previous results.
\end{abstract}

\section{Introduction}

The design of robust networks consists in   improving existing networks so that the resulting networks resist to different types of failures. A typical  problem is the global edge-connectivity augmentation problem where the goal is  to add a minimum number of edges to a given undirected graph to obtain a graph that remains connected after deleting any set of edges of a given size. Watanabe and Nakamura \cite{WANA}, Cai and Sun \cite{CASU} independently solved the problem in the sense of a minimax theorem and  an efficient algorithm. Frank \cite{FA92} developed a method to solve edge-connectivity augmentation problems in general which for instance provided the solution of the local edge-connectivity augmentation problem. His paper has stimulated further research in a great number of directions that led to many interesting generalizations. For a survey on the subject see \cite{SZAUG}. 

Let us now consider the directed case. In  \cite{FA92}, Frank solved the global arc-connectivity augmentation problem and proved that the local arc-connectivity augmentation problem  is NP-complete. Further, in \cite{book}, Frank solved the rooted $k$-arc-connectivity augmentation problem, that is when the local arc-connectivity requirement is $k$ from a given vertex $s$ to all the other vertices and $0$ otherwise. By the fundamental theorems of Edmonds \cite{Egy} and Menger \cite{Ketto}, this special case is equivalent to the augmentation problem for packing  $k$ spanning $s$-arborescences. In this paper we study some more complex arborescence packing augmentation problems. We mention that for any arborescence packing problem, the augmentation version extends the original problem. On the one hand, we propose the solution of the augmentation version of $h$-regular  {\sf M}-independent-rooted $(f,g)$-bounded $(\alpha,\beta)$-limited packing of mixed hyperarborescences. For the definitions see Section \ref{secdef}. Our solution depends on the discovery of a new submodular function and on the theory of generalized polymatroids. This way we are able to unify the results of two previous papers \cite{szighyp} and \cite{szigrooted} and to provide a simpler proof of this common extension. On the other hand, we propose the solution of the augmentation version of  $h$-regular $(\ell, \ell')$-bordered $(\alpha, \beta)$-limited packing of $k$ hyperbranchings, and also its undirected counterpart. Our solution depends on a new augmentation lemma and the recent results of \cite{hopmarszig}.

Our theorems generalize the results of the following papers: \cite{BF3}, \cite{cai1}, \cite{Egy}, \cite{FKLSzT}, \cite{FA78}, \cite{fkiki}, \cite{fkk},  \cite{gao}, \cite{gy2}, \cite{hopmarszig}, \cite{HSz5}, \cite{NW}, \cite{PCK},  \cite{szighyp}, \cite{szigrooted}, \cite{Tu}.

%
\section{Definitions}\label{secdef}

We use the usual notation for the set $\pmb{\mathbb{Z}}$ of integers, for the set $\pmb{\mathbb{Z}_+}$ of non-negative integers and $\pmb{\mathbb{Z}_k}=\{1,\dots,k\}.$ An inequality is called {\it tight} if it holds with equality, otherwise it is  {\it strict}. Let $V$ be a finite set. For a function $m: V\rightarrow \mathbb{Z}$ and a subset $X$ of $V,$ we define $m(X)=\sum_{x\in X}m(x).$ For a function $g:V\rightarrow \mathbb{Z}_+$ and a non-negative integer $h,$ the function {\boldmath $g_h$} is defined as follows: $g_h(v)=\min\{g(v),h\}$ for every $v\in V.$ The function {\boldmath $\infty_0$} on $V$ has value $\infty$ everywhere except for the emptyset where we set its value to $0.$ For a subset $X$ of $V,$ its {\it complement} is denoted by {\boldmath $\overline X$}. A {\it multiset}  of $V$ is a set of  elements of $V$  taken with multiplicities. 
For a multiset $S$ of $V$ and a subset $X$ of $V$, we denote by {\boldmath$S_X$} the multiset of $V$ which is the restriction of $S$ on $X.$ For $k\in\mathbb{Z}_+,$ we denote by {\boldmath$k\times V$} the multiset of $V$ where every element of $V$ is taken with multiplicity $k.$ A set ${\cal S}$ of subsets of $V$ is called a {\it family} if the subsets of $V$ are taken with multiplicities in ${\cal S}$. For a family ${\cal S}$ of subsets of $V$ and  a subset $X$ of $V$, we denote by {\boldmath${\cal S}_X$} the subfamily of ${\cal S}$ containing the sets in ${\cal S}$ that intersect $X$. We say that two subsets $X$ and $Y$ of $V$ are {\it properly intersecting} if none of $X\cap Y, X - Y,$ and $Y-X$ is empty. We mean by {\it uncrossing} two properly intersecting sets the operation that replaces the two sets by their intersection and their union.  

A set of mutually disjoint subsets of $V$ is called a {\it subpartition}. If ${\cal P}$ is a subpartition of $V$, then {\boldmath$\cup {\cal P}$} denotes the union of the sets in ${\cal P}$. If  ${\cal P}$ is a subpartition of $V$ and $\cup {\cal P}=V$, then ${\cal P}$ is called a {\it partition}. Let  ${\cal P}_1$ and ${\cal P}_2$ be two subpartitions of $V$, $P_1=\cup {\cal P}_1,$  $P_2=\cup {\cal P}_2,$ and ${\cal P}={\cal P}_1\cup {\cal P}_2.$  Note that ${\cal P}$ covers each element in $P_1\cap P_2$ twice and each element in $(P_1\cup P_2)-(P_1\cap P_2)$ once. Using the usual uncrossing method on ${\cal P}$, we obtain a family ${\cal P}'$ that contains no properly intersecting sets and  that covers each element in $P_1\cap P_2$ twice and each element in $(P_1\cup P_2)-(P_1\cap P_2)$ once. Then, by taking respectively the minimal  and maximal  sets in ${\cal P}'$, we obtain  a partition ${\cal P}'_1$ of $P_1\cap P_2$ and a partition ${\cal P}'_2$ of $P_1\cup P_2.$  We mention that while ${\cal P}'_1$ depends on the particular execution of the uncrossing method, $\cup{\cal P}'_1$, $|{\cal P}'_1|$, and ${\cal P}'_2$ are uniquely defined. We define ${\cal P}'_1$ as the intersection {\boldmath${\cal P}_1\sqcap {\cal P}_2$} of ${\cal P}_1$ and ${\cal P}_2$, and ${\cal P}'_2$ as the union {\boldmath${\cal P}_1\sqcup {\cal P}_2$} of ${\cal P}_1$ and ${\cal P}_2$. We will only use   properties on ${\cal P}_1\sqcap {\cal P}_2$ and ${\cal P}_1\sqcup {\cal P}_2$ that are true for every execution of the uncrossing method such as:
\begin{align}
&	\text{If $U_1 \in \mathcal{P}_1$ and $U_2 \in \mathcal{P}_2$ intersect, then an element of $\mathcal{P}_1 \sqcap \mathcal{P}_2$ contains $U_1 \cap U_2$.} \label{kvkjvk1}	\\
&	\text{If $U \in \mathcal{P}_1 \cup \mathcal{P}_2$, then an element of  $\mathcal{P}_1 \sqcup \mathcal{P}_2$ contains $U$.} \label{kvkjvk2}	\\
&	|\mathcal{P}_1| + |\mathcal{P}_2| = |\mathcal{P}_1 \sqcap \mathcal{P}_2| + |\mathcal{P}_1 \sqcup \mathcal{P}_2|.\label{kvkjvk3}
\end{align}

Let $S$ be a finite ground set. A set function $b$ on $S$ is called {\it non-decreasing} if $b(X)\le b(Y)$ for all $X\subseteq Y\subseteq S,$  {\it subcardinal} if $b(X)\le |X|$ for every $X\subseteq S,$ and   {\it submodular} if  $b(X)+b(Y) 	\geq  b(X\cap Y)+b(X\cup Y)$ for  all $X,Y\subseteq S.$ A set function $p$ on $S$ is called {\it supermodular} if $-p$ is submodular. A set function $m$ on $S$ is called {\it modular} if it is submodular and supermodular. Let $r$ be a non-negative integer-valued function on $S$ such that  $r$ is subcardinal, non-decreasing and submodular. Then {\boldmath${\sf M}$} $=(S,r)$ is called a {\it matroid}. The function $r$ is called the {\it rank function} of the matroid ${\sf M}.$ If a matroid ${\sf M}$ is given, then we denote its rank function by {\boldmath$r_{\sf M}$}. An {\it independent set} of ${\sf M}$ is a subset $X$ of $S$ such that $r_{\sf M}(X)=|X|.$  The set of independent sets of ${\sf M}$ is denoted by {\boldmath${\cal I}_{\sf M}$}. A maximal independent set of ${\sf M}$ is called a {\it basis}. The {\it free matroid} is the matroid where every subset of $S$ is independent. The {\it uniform matroid} $U_{S,k}$ of rank $k$ is the matroid whose independent sets are the subsets of $S$ of size at most $k.$ 

Let {\boldmath$D$} $=(V,A)$ be a directed graph or {\it digraph} with {\it vertex set}  {\boldmath$V$}  and   {\it arc set} {\boldmath$A$}. An {\it arc} $e=uv$ is an ordered pair of different vertices $u$ (the {\it tail} of  $e$) and $v$ (the {\it head} of  $e$).  
For a subset $X$ of $V,$ the set of arcs in $A$ {\it entering $X$}, that is their heads are in $X$ and their tails are in $\overline X$, is denoted by {\boldmath$\delta^-_A(X)$}. The {\it in-degree} of $X$ is {\boldmath$d^-_A(X)$} $=|\delta^-_A(X)|.$ A digraph $F=(U,B)$ is called an {\it $s$-arborescence} if $s\in U$ and every vertex of $F$ is reachable from $s$ via a unique path in $F.$  The vertex $s$ is called the {\it root} of the $s$-arborescence.  We say that $F$ is  an {\it $S$-branching} if $S\subseteq U$ and there exists a unique path from $S$ to every $v\in U$ in $F.$ The vertex set $S$ is called the {\it root set} of the $S$-branching. A branching $F$ is called a {\it spanning branching} of $D$ if $U=V$ and $B\subseteq A$. Note that if $S=\{s\}$, then an $S$-branching is an $s$-arborescence. For non-negative integer-valued functions $f$ and $g$ on $V$, an arc set $F$ is called {\it $(f,g)$-indegree-bounded} if $f(v)\le d_F^-(v)\le g(v)$ for every $v\in V.$ For non-negative integers $q$ and $q'$,  an arc set $F$ is called {\it $(q,q')$-size-limited} if $q\le |F|\le q'.$ 

 Let {\boldmath$\mathcal{D}$} $=(V,\mathcal{A})$ be a directed hypergraph or {\it dypergraph} with dyperedge set  {\boldmath$\mathcal{A}$}. A {\it dyperedge} $e=Zz$ is an ordered pair of a non-empty subset $Z$ of $V-z$ (the set of {\it tails}  of $e$) and a vertex $z$ in $V$ (the  {\it head} of $e$). For a  subset $X$ of $V,$ a dyperedge $Zz$ {\it enters $X$} if $z\in X$ and $Z\cap \overline X\neq\emptyset.$ The  set of dyperedges in $\mathcal{A}$ {\it entering $X$} is denoted by {\boldmath$\delta^-_\mathcal{A}(X)$}  and the {\it in-degree} of $X$ is {\boldmath$d^-_\mathcal{A}(X)$} $=|\delta^-_\mathcal{A}(X)|.$ By {\it trimming} a dyperedge $Zz$, we mean the operation that replaces $Zz$ by an arc $yz$ for some $y\in Z.$ We say that $\mathcal{F}$ is  an {\it $S$-hyperbranching} if $\mathcal{F}$ can be trimmed to an  $S$-hyperbranching. If $S=\{s\},$ then an  $S$-branching is called an {\it $s$-hyperarborescence}.  An  $S$-hyperbranching $\mathcal{F}=(U,\mathcal{B})$ is called a \textit{spanning $S$-hyperbranching} of $\mathcal{D}$ if $U=V$, $\mathcal{B}\subseteq \mathcal{A}$, and $|S|+|\mathcal{B}|=|V|.$ 

Let {\boldmath$\mathcal{F}$} $=(V,\mathcal{E}\cup \mathcal{A})$ be a {\it mixed hypergraph} with hyperedge set {\boldmath$\mathcal{E}$}  and dyperedge set {\boldmath$\mathcal{A}$}. A {\it hyperedge} is a subset $Z$ of $V$ containing at least two distinct elements. For a  subset $X$ of $V,$  a hyperedge $Z$ {\it enters}  $X$ if $Z\cap X\neq\emptyset\neq Z\cap \overline X.$ By {\it orienting} a hyperedge $Z\in \mathcal{E},$ we mean the operation that replaces the hyperedge $Z$ by a dyperedge $Z'z$ where $z\in Z$ and $Z'=Z-z.$ An {\it orientation} of $\mathcal{F}$ is obtained from $\mathcal{F}$ by orienting every hyperedge in $\mathcal{E}$. A {\it mixed (spanning) $S$-hyperbranching} is a mixed hypergraph  that has an orientation that is a (spanning)  $S$-hyperbranching. In particular, if $S=\{s\},$ then we are speaking of a mixed (spanning)  $s$-hyperarborescence. 
A mixed $S$-hyperbranching is called a {\it rooted $S$-hyperforest} if it contains no dyperedge. For a subpartition ${\cal P}$ of $V$, we denote by {\boldmath$e_{\mathcal{E}\cup \mathcal{A}}({\cal P})$} the number of hyperedges in $\mathcal{E}$ and dyperedges in $\mathcal{A}$ that enter at least one member of ${\cal P}$. By a {\it packing} of mixed hyperbranchings in $\mathcal{F}$, we mean a set of mixed hyperbranchings that are hyperedge- and dyperedge-disjoint.

Let ${\cal B}$ be a packing of arborescences in a digraph $D$. For a positive integer $h$, the packing ${\cal B}$ is called {\it $h$-regular} if each vertex of $D$ belongs to exactly $h$ arborescences in ${\cal B}$. For non-negative integer-valued functions $f$ and $g$ on $V$, the packing ${\cal B}$ is called {\it $(f,g)$-bounded} if the number of arborescences in ${\cal B}$ rooted at $v$ is at least $f(v)$ and at most $g(v)$ for every vertex $v$ of $D.$ For non-negative integers $\alpha$ and $\beta$, the packing ${\cal B}$ is called {\it $(\alpha,\beta)$-limited} if the number of arborescences in ${\cal B}$ is at least $\alpha$ and at most $\beta.$ For a multiset $S$ of vertices in $V$ and a matroid {\sf M} on $S$, the packing ${\cal B}$ is called {\it {\sf M}-independent-rooted} if the root set of the arborescences in ${\cal B}$ forms an independent set in {\sf M}. If the root set of the arborescences in ${\cal B}$ forms a basis in {\sf M}, then  the packing ${\cal B}$ is called {\it {\sf M}-basis-rooted}. 

Let $\mathcal{D}$ be a dypergraph,  $\mathcal{F}$  a mixed hypergraph and  {\sf  P} a subset of the properties of {\sf M}-independent-rooted, {\sf M}-bases-rooted,  $(f,g)$-bounded, $h$-regular, and $(\alpha,\beta)$-limited. We say that $\mathcal{D}$ has a {\it {\sf  P} packing} of hyperarborescences if $\mathcal{D}$  can be trimmed to a digraph that has a {\sf  P} packing of arborescences. We say that $\mathcal{F}$ has a {\it {\sf  P} packing} of mixed hyperarborescences if $\mathcal{F}$  can be oriented to a dypergraph that has a {\sf  P} packing of hyperarborescences. For $\ell, \ell' : \mathbb{Z}_k \rightarrow \mathbb{Z}_+$, a packing of $k$ hyperbranchings  (rooted hyperforests) with root sets $S_1,\dots,S_k$  in $\mathcal{F}$  is said to be \textit{$(\ell, \ell')$-bordered}  if  $\ell(i)\le |S_i| \le \ell'(i)$ for every $1 \le i \le k.$

\section{Generalized polymatroids}

We present the necessary definitions and results from the theory of generalized polymatroids. In this section let $p$ and $b$ be a supermodular and a submodular set function on $S$ such that $p(\emptyset)=0=b(\emptyset)$. Let $f$ and $g$  be non-negative integer-valued functions on $S$ and $\alpha$ and $\beta$ non-negative integers. We will use the following polyhedra.
 \begin{eqnarray*}
\text{\boldmath$Q(p,b)$} 			&	=	&	\{x\in \mathbb{R}^S: p(Z)\le x(Z)\le b(Z)\ \text{ for all } Z\subseteq S\},	\\
  \text{\boldmath$T(f,g)$}			& 	=	&	\{x\in \mathbb{R}^S:  f(s)\le x(s)\le g(s)\ \text{ for all }  s\in S\},			\\
\text{\boldmath$K(\alpha,\beta)$} 	&	=	&	\{x\in \mathbb{R}^S: \alpha\le x(S)\le \beta\},						\\
\text{\boldmath$Q_1+Q_2$} 		&	=	&	\{x_1+x_2: x_1\in Q_1,x_2\in Q_2\}.
\end{eqnarray*}

If  $b(X)-p(Y )\ge b(X-Y )-p(Y-X)$ for all $X,Y\subseteq S,$ then $Q(p,b)$ is called a {\it generalized polymatroid}, shortly {\it $g$-polymatroid}.
\medskip

We need the following results on generalized polymatroids.

\begin{thm}[Frank \cite{book}]\label{gpm}
The following hold.
\begin{enumerate}[itemsep=0cm]
	\item \label{TQ} $T(f,g)\neq\emptyset$ if and only if $f\le g.$ If $T(f,g)\neq\emptyset,$   then it is a g-polymatroid $Q(f,g).$ 
		\item \label{QcapT} $Q(p,b)\cap T(f,g)\neq\emptyset$  if and only if $\max\{p,f\}\le\min\{b,g\}.$
 		If $Q(p,b)\cap T(f,g)\neq\emptyset$, then it is a g-polymatroid $Q(p^g_f, b^g_f)$ with
		\begin{eqnarray*}
			\text{{\boldmath $p^g_f(Z)$}} &=&\max \{p(X)-g(X-Z)+f(Z-X): X\subseteq S\},\label{interboxp}\\
			\text{{\boldmath $b^g_f(Z)$}} &=&\min\ \{b(X)-f(X-Z)+g(Z-X): X\subseteq S\}.\label{interboxb}
		\end{eqnarray*}
	\item \label{QcapK} $Q(p,b)\cap K(\alpha,\beta)\neq\emptyset$ if and only if $p\le b$, $\alpha\le \beta,$ 
		$\beta\ge p(S)$ and $\alpha\le b(S).$ If $Q(p,b)\cap K(\alpha,\beta)\neq\emptyset$, then  it is a 
		g-polymatroid  $Q(p_\alpha^\beta, b_\alpha^\beta)$ with
		\begin{eqnarray*}
			\text{{\boldmath $p_\alpha^\beta(Z)$}} &=&\max\{p(Z), \alpha-b(\overline Z)\}, \\
		\ \ \text{{\boldmath $b_\alpha^\beta(Z)$}} &=&\min\{b(Z), \beta-p(\overline Z)\}.
		\end{eqnarray*}
	\item \label{sum} $Q(p_1,b_1)+Q(p_2,b_2)=Q(p_1+p_2,b_1+b_2).$
		If $p_1,b_1,p_2,b_2$ are integral, then every integral element  of $Q(p_1+p_2,b_1+b_2)$ arises as 	
		the sum of an integral element of $Q(p_1,b_1)$ and an integral element of $Q(p_2,b_2)$.
	\item \label{QcapQ} $Q(p_1,b_1)\cap Q(p_2,b_2)\neq\emptyset$  if and only if $p_1\le b_2$ and $p_2\le b_1$. 
		If $p_1,b_1,p_2,b_2$ are integral and the intersection is not empty, then it contains an integral element.
\end{enumerate}
\end{thm}

\section{Packing problems}\label{Packingproblems}

This section lists some results on packing arborescences or more generally on packing mixed hyperbranchings relevant to this paper. It contains three subsections: the first one on packing mixed arborescences, the second one on packing mixed hyperarborescences, and the last one on packing mixed hyperbranchings.

\subsection{Packing arborescences and  mixed arborescences}

We start our list by the classic result of Edmonds \cite{Egy} on packing arborescences with fixed roots.

\begin{thm}[Edmonds \cite{Egy}]\label{edmondsarborescencesmulti}
Let $D=(V,A)$ be a digraph and $S=\{s_1,\dots,s_k\}$ a multiset of vertices in $V.$  
There exists a packing of $k$ spanning arborescences with roots $s_1,\dots,s_k$ in $D$ if and only if 
	\begin{eqnarray*}
		|S_X|+d^-_A(X)&\geq &k \hskip .5truecm \text{ for every non-empty  $X\subseteq V.$}
	\end{eqnarray*}
\end{thm}

The next result  on packing arborescences with flexible roots is due to Frank \cite{FA78}.

\begin{thm}[Frank \cite{FA78}]\label{frankarborescences} 
Let $D=(V,A)$ be a digraph and $k\in \mathbb{Z}_+.$
 There exists a packing of   $k$  spanning arborescences in $D$ if and only if 
 	\begin{eqnarray*}  
		e_A({\cal P})+k &\geq &k|{\cal P}| \hskip .5truecm \text{ for every subpartition ${\cal P}$ of  $V.$}
	\end{eqnarray*}
\end{thm}

It is well-known that  Theorems \ref{edmondsarborescencesmulti}  and \ref{frankarborescences} are equivalent.
\medskip

Theorem \ref{frankarborescences} was generalized for $(f,g)$-bounded packings as follows.

\begin{thm}[Frank \cite{FA78}, Cai \cite{cai1}]\label{frankarborescencesroots} 
Let $D=(V,A)$ be a digraph, $f,g: V\rightarrow \mathbb Z_+$ functions and $k\in \mathbb{Z}_+.$
There exists an $(f,g)$-bounded packing of   $k$  spanning arborescences in $D$  if and only if 
\begin{eqnarray}
	g(v)			& 	\geq 		&	f(v)\hskip .52truecm  \text{ for every } v \in V,\label{fg} \\ 
	e_A({\cal P})+\min\{k-f(\overline{\cup\mathcal{P}}), g({\cup\mathcal{P}})\}		&	\ge 	& 	k|\mathcal{P}| 
		\hskip .52truecm \text{ for every subpartition } \mathcal{P} \text{ of } V.
\end{eqnarray}
\end{thm}

For $f(v)=|S_v|=g(v)$ for all $v\in V$ and $k=|S|$ for a multiset $S$ of vertices in $V$, Theorem \ref{frankarborescencesroots}  reduces to Theorem \ref{edmondsarborescencesmulti}. For $f(v)=0$ and $g(v)=k$ for all $v\in V$, Theorem \ref{frankarborescencesroots}  reduces to Theorem \ref{frankarborescences}.
\medskip

A special case of a theorem of B\'erczi and Frank \cite{BF3} provides the following extension of Theorem \ref{frankarborescencesroots}.

\begin{thm}[B\'erczi, Frank \cite{BF3}]\label{bobisuv}
Let $D=(V,A)$ be a digraph, $f,g: V\rightarrow \mathbb Z_+$ functions and $h,\alpha,\beta\in\mathbb Z_+$. 
There is an $h$-regular $(f,g)$-bounded $(\alpha,\beta)$-limited packing of arborescences in $D$  if and only if   
	\begin{eqnarray}  
		g_h(v)		& 	\ge	&	f(v)	\hskip .52truecm  \text{ for every } v \in V,		\label{fghu} \\ 
	\min\{\beta,g_h(V)\}	&	\ge 	& 	\alpha,									\label{iviytdtutuy}\\
	e_A({\cal P})+\min\{\beta-f(\overline{\cup\mathcal{P}}), g({\cup\mathcal{P}})\} 		&	\ge 	& 
	h|\mathcal{P}|	\hskip .52truecm \text{ for every subpartition } \mathcal{P} \text{ of } V.
	\end{eqnarray}
\end{thm}

For $h=\alpha=\beta=k,$ Theorem \ref{bobisuv} reduces to Theorem \ref{frankarborescencesroots}. 
\medskip

Theorems \ref{frankarborescences} and \ref{frankarborescencesroots} were extended to mixed graphs as follows.

\begin{thm}[Frank \cite{FA78}]\label{gaoyang2}
Let $F=(V,E\cup A)$ be a mixed graph and $k\in \mathbb{Z}_+.$
There exists a packing of $k$  spanning mixed  arborescences in  $F$ if and only if 
\begin{eqnarray} \label{condmixedflex}
	e_{E\cup A}({\cal P}) &\geq& k(|{\cal P}|-1) \hskip .44truecm \text { for every subpartition } {\cal P} \text { of } V.
\end{eqnarray}
\end{thm}

If $E=\emptyset$, then Theorem \ref{gaoyang2} reduces to Theorem \ref{frankarborescences}.

\begin{thm}[Gao,Yang \cite{gy2}]\label{gaoyang3} 
Let $F=(V,E\cup A)$ be a mixed graph, $f,g: V\rightarrow \mathbb Z_+$ functions, and $k\in \mathbb{Z}_+.$
There exists an $(f,g)$-bounded  packing of $k$  spanning mixed arborescences in $F$ if and only if \eqref{fg} is satisfied and  \begin{eqnarray}\label{jvljh1}
	e_{E\cup A}({\cal P})+\min\{k-f(\overline{\cup\mathcal{P}}),g(\cup\mathcal{P})\}		&	\ge 	& 	k|\mathcal{P}|
	\hskip .52truecm \text{ for every subpartition } \mathcal{P} \text{ of } V.
\end{eqnarray}
\end{thm}

If $E=\emptyset$, then Theorem \ref{gaoyang3} reduces to Theorem \ref{frankarborescencesroots}. For $f(v)=0$ and $g(v)=k$ for every $v\in V,$ Theorem \ref{gaoyang3} reduces to Theorem \ref{gaoyang2}.

\subsection{Packing hyperarborescences and mixed hyperarborescences}

Theorem \ref{edmondsarborescencesmulti} was  extended for dypergraphs by Frank, Kir\'aly, and Kir\'aly \cite{fkiki}.

\begin{thm}[Frank, Kir\'aly, Kir\'aly \cite{fkiki}]\label{hyperarborescencesmulti} 
Let $\mathcal{D}=(V,\mathcal{A})$ be a dypergraph and $S=\{s_1,\dots,s_k\}$ a multiset of vertices in $V.$
There exists a packing of $k$ spanning hyperarborescences  with roots $s_1,\dots,s_k$ in $\mathcal{D}$ if and only if 
	\begin{eqnarray} \label{fkkcondmulti} 
		|S_X|+	 d^-_\mathcal{A}(X) &\geq & k \hskip .5truecm \text{ for every  non-empty  $X\subseteq V.$}
	\end{eqnarray}
\end{thm}

If $\mathcal{D}$ is a digraph, then Theorem \ref{hyperarborescencesmulti} reduces to Theorem \ref{edmondsarborescencesmulti}.
\medskip

H\"orsch and Szigeti \cite{HSz5} extended Theorem \ref{gaoyang3} to mixed hypergraphs.

\begin{thm}[H\"orsch, Szigeti \cite{HSz5}]\label{hsz} 
Let $\mathcal{F}=(V,\mathcal{E}\cup \mathcal{A})$ be a mixed hypergraph, $f,g: V\rightarrow \mathbb Z_+$ functions, and $k\in \mathbb{Z}_+.$
There exists an $(f,g)$-bounded  packing of $k$  spanning mixed hyperarborescences in $\mathcal{F}$ if and only if \eqref{fg} is satisfied and  \begin{eqnarray}\label{jvljh1hyp}
	e_{\mathcal{E}\cup \mathcal{A}}({\cal P})+\min\{k-f(\overline{\cup\mathcal{P}}),g(\cup\mathcal{P})\}		&	\ge 	& 	
	k|\mathcal{P}|\hskip .52truecm \text{ for every subpartition } \mathcal{P} \text{ of } V.
\end{eqnarray}
\end{thm}

If $\mathcal{F}$ is a mixed graph, then Theorem \ref{hsz} reduces to Theorem \ref{gaoyang3}. If $\mathcal{E}=\emptyset$, $k=|S|,$ and $f(v)=g(v)=|S_v|$ for all $v\in V,$ then Theorem \ref{hsz} reduces to Theorem \ref{hyperarborescencesmulti}.
\medskip

An extension with a matroid constraint of a common generalization of Theorems \ref{frankarborescencesroots} and \ref{hyperarborescencesmulti} was given in Szigeti \cite{szigrooted}.

\begin{thm}[Szigeti \cite{szigrooted}]\label{jegdvvlmcblkcbkuhyper}
Let $\mathcal{D}=(V,\mathcal{A})$ be a dypergraph, {$h$} $\in \mathbb{Z}_+$, $f,g\in \mathbb Z^V_+$,  {$S$} a multiset of vertices in $V$, and {${\sf M}$} $=(S,r_{\sf M})$ a matroid. There exists an  $h$-regular ${\sf M}$-basis-rooted $(f,g)$-bounded  packing of hyperarborescences in $\mathcal{D}$ if and only if  \eqref{fghu} holds and  for all $X,Z\subseteq V$ and subpartition $\mathcal{P}$ of $Z,$
\begin{eqnarray}
	r_{\sf M}(S_X)+g_h(\overline X)							&	\ge	& 	r_{\sf M}(S),\label{bsjvuhcyqcuq2}\\
	e_{\mathcal{A}}(\mathcal{P})+r_{\sf M}(S_X)-f(X-Z)+g_h(Z-X)	&	\ge	& 	h|\mathcal{P}|.
\end{eqnarray}
\end{thm}

If $\mathcal{D}$ is digraph, ${\sf M}=U_{k\times V,k}$ is the uniform matroid of rank $k$ on $k\times V$ and $h=k$, then Theorem \ref{jegdvvlmcblkcbkuhyper} reduces to Theorem \ref{frankarborescencesroots}. For a multiset $S$ of $V$, if ${\sf M}$ is the free matroid on $S$, $f(v)=g(v)=|S_v|$ for every $v\in V,$ and $h=|S|$, then Theorem \ref{jegdvvlmcblkcbkuhyper} reduces to Theorem \ref{hyperarborescencesmulti}.
 \medskip
 
Theorem \ref{jegdvvlmcblkcbkuhyper} and an orientation result of Gao \cite{gao} provide the following theorem.

\begin{thm}\label{jegdvvlmcblkcbkuhypermix}
Let $\mathcal{F}=(V,\mathcal{E}\cup\mathcal{A})$ be a mixed hypergraph, {$h$} $\in \mathbb{Z}_+$, $f,g\in \mathbb Z^V_+$,  {$S$} a multiset of vertices in $V$, and {${\sf M}$} $=(S,r_{\sf M})$ a matroid. There exists an $h$-regular  ${\sf M}$-basis-rooted $(f,g)$-bounded  packing of mixed hyperarborescences in $\mathcal{F}$ if and only if \eqref{fghu} and \eqref{bsjvuhcyqcuq2} hold and for all $U,W\subseteq V$ and subpartition $\mathcal{P}$ of $W,$
\begin{eqnarray}\label{bsjvuhcyqcuqhypermix}
	e_{\mathcal{E}\cup\mathcal{A}}(\mathcal{P})+r_{\sf M}(S_{U})-f(U-W)+g_h(W-U)	&	\ge	& 	h|\mathcal{P}|.
\end{eqnarray}
\end{thm}

If $\mathcal{F}$ is  a mixed graph, {\sf M} is the uniform matroid of rank $k$ on $k\times V$ and $h=k$, then Theorem \ref{jegdvvlmcblkcbkuhypermix} reduces to Theorem \ref{gaoyang3}. If $\mathcal{F}$ is a dypergraph, then Theorem \ref{jegdvvlmcblkcbkuhypermix} reduces to Theorem \ref{jegdvvlmcblkcbkuhyper}.
\medskip

An extension for mixed hypergraphs of a common generalization of Theorems \ref{bobisuv} and \ref{hsz} was given in Szigeti \cite{szighyp}.

\begin{thm}[Szigeti \cite{szighyp}]\label{sibevhz2} 
Let $\mathcal{F}=(V,\mathcal{E}\cup \mathcal{A})$ be a mixed hypergraph, $f,g: V\rightarrow \mathbb Z_+$ functions, and $h,\alpha,\beta\in \mathbb{Z}_+.$
There exists an $h$-regular $(f,g)$-bounded $(\alpha,\beta)$-limited  packing of mixed hyperarborescences in $\mathcal{F}$ if and only if \eqref{fghu} and \eqref{iviytdtutuy}    hold and  
\begin{eqnarray}\label{jvljh1hyp2}
	e_{{\cal E}\cup{\cal A}}({\cal P})+\min\{\beta-f(\overline{\cup{\cal P}}),g_h(\cup{\cal P})\}		&	\ge 		& 	h|{\cal P}|
	 \hskip .52truecm \text{ for every subpartition } {\cal P} \text{ of } V.
\end{eqnarray}
\end{thm}

If $\mathcal{F}$ is digraph, then Theorem \ref{sibevhz2} reduces to Theorem \ref{bobisuv}. If  $h=\alpha=\beta=k,$ then Theorem \ref{sibevhz2} reduces to Theorem \ref{hsz}. 
\medskip

As a new result we now provide a common generalization of the previous two theorems.

\begin{thm}\label{bjhkcgckh} 
Let $\mathcal{F}=(V,\mathcal{E}\cup\mathcal{A})$  be a mixed hypergraph, $f,g: V\rightarrow \mathbb Z_+$ functions, $h,\alpha,\beta\in \mathbb{Z}_+,$ {$S$} a multiset of vertices in $V$, and {${\sf M}$} $=(S,r_{\sf M})$ a matroid. 
There exists an $h$-regular ${\sf M}$-independent-rooted $(f,g)$-bounded $(\alpha,\beta)$-limited packing of mixed hyperarborescences in $\mathcal{F}$ if and only if \eqref{fghu}   holds and  for all  $X, Z\subseteq V$ and subpartition $\mathcal{P}$ of $Z$,
\begin{eqnarray}
	\alpha											&	\le 	& 	\beta,	\label{ellll}\\
	\max\{h,\alpha\}-r_{\sf M}(S_{\overline X})+ f(Z-X)-g_h(X-Z)	&	\le	&	h|Z|,		\label{utrdciyvouicomgenj}\\
	e_{{\cal E}\cup{\cal A}}({\cal P})+\min\{\beta-f(\overline Z),r_{\sf M}(S_{X})-f(X-Z)+g_h(Z-X)\}	&	\ge 	& 	h|\mathcal{P}|.			\label{uhvigcftfcomgen}
\end{eqnarray}
\end{thm}

If $\alpha=\beta=r_{\sf M}(S)$, then Theorem \ref{bjhkcgckh} reduces to Theorem \ref{jegdvvlmcblkcbkuhypermix}. If {\sf M} is the free matroid on $h\times V$, then Theorem \ref{bjhkcgckh} reduces to Theorem \ref{sibevhz2}. Theorem \ref{bjhkcgckh} will easily follow from  Theorem \ref{iycivsxcomgen}.

\subsection{Packing branchings and mixed hyperbranchings}

Edmonds \cite{Egy} also gave the characterization of the existence of a packing of spanning branchings with fixed root sets.

\begin{thm}[Edmonds \cite{Egy}]\label{edmondsbranchings}
Let $D=(V,A)$ be a digraph and ${\cal S}=\{S_1,\dots,S_k\}$ a family of subsets of $V$.  There exists a packing of $k$ spanning branchings with root sets $S_1,\dots,S_k$ in $D$ if and only if 
	\begin{eqnarray*}
		|{\cal S}_X|+d^-_A(X)&\geq &k \hskip .5truecm \text{ for every non-empty  $X\subseteq V.$}
	\end{eqnarray*}
\end{thm}

For ${\cal S}=\{s:s\in S\},$ Theorem \ref{edmondsbranchings} reduces to Theorem \ref{edmondsarborescencesmulti}.
\medskip

Another special case of a theorem of B\'erczi and Frank \cite{BF3} provides the following extension of Theorem \ref{edmondsarborescencesmulti}. Let us recall that for a function $\ell: \mathbb{Z}_k \rightarrow \mathbb{Z}_+$, $\ell(\mathbb{Z}_k)=\sum_{i=1}^k\ell(i).$

\begin{thm}[B\'erczi, Frank \cite{BF3}]\label{BFmain}
	Let $D = (V, A)$ be a digraph, $k, \alpha, \beta \in \mathbb{Z}_+$, and $\ell, \ell': \mathbb{Z}_k \rightarrow \mathbb{Z}_+$ such that
	\begin{alignat}{3}
		\ell'(\mathbb{Z}_k) 	&\ge \beta \ge \alpha \ge \ell(\mathbb{Z}_k), 				\label{totnecessary}\\
		|V| 				&\ge \ell'(i) \ge \ell(i) \qquad\qquad \text{for every } 1 \le i \le k. 	\label{indivnecessary}
	\end{alignat}
	There exists an $(\ell, \ell')$-bordered $(\alpha, \beta)$-limited packing of $k$ spanning  branchings in $D$ if and only if
	\begin{alignat}{3}
	\beta - \ell(\mathbb{Z}_k) + \sum_{i=1}^k\min\{|\mathcal{P}|,\ell(i)\} + e_A(\mathcal{P}) 	&	\ge k|\mathcal{P}| \qquad 		&&	\text{for every subpartition } \mathcal{P} \text{ of } V, \\
	\sum_{i=1}^k\min\{|\mathcal{P}|,\ell'(i)\} + e_A({\cal P}) 	&	\ge k|\mathcal{P}| \qquad 
	&&	\text{for every subpartition } \mathcal{P} \text{ of } V.
	\end{alignat}
\end{thm}

It was mentioned in \cite{hopmarszig} that Theorem \ref{BFmain} can be generalized as follows to $h$-regular packings in dypergraphs.

\begin{thm}[Hoppenot, Martin, Szigeti \cite{hopmarszig}]\label{jjjnbbbnajkhypregdir}
Let $\mathcal{D} = (V, \mathcal{A})$ be a dypergraph, $h, k, \alpha, \beta \in \mathbb Z_+$ and $\ell, \ell': \mathbb{Z}_k \rightarrow  \mathbb{Z}_+$ such that~\eqref{totnecessary} and~\eqref{indivnecessary} hold. There exists an $h$-regular $(\ell, \ell')$-bordered $(\alpha, \beta)$-limited packing of $k$  hyperbranchings in $\mathcal{D}$ if and only if
\begin{alignat}{3}
	h|V| 	&	\ge \alpha, \label{hValpha}\\
	\beta - \ell(\mathbb{Z}_k) + \sum_{i=1}^k\min\{|\mathcal{P}|,\ell(i)\} + e_\mathcal{A}(\mathcal{P}) 	&	\ge h|\mathcal{P}| 	\qquad 	&&	\text{for every subpartition } \mathcal{P} \text{ of } V, \label{konfeiobfiueguye1mregdir}\\
	\sum_{i=1}^k\min\{|\mathcal{P}|,\ell'(i)\} + e_\mathcal{A}(\mathcal{P}) 	&	\ge h|\mathcal{P}| \qquad 
			&&\text{for every subpartition } \mathcal{P} \text{ of } V. \label{konfeiobfiueguye2mregdir}
	\end{alignat}
\end{thm}

If $\mathcal{D}$ is a digraph and $h = k$, then Theorem \ref{jjjnbbbnajkhypregdir} reduces to Theorem \ref{BFmain}.
\medskip

The undirected counterpart of Theorem \ref{jjjnbbbnajkhypregdir} follows.

\begin{thm}[Hoppenot, Martin, Szigeti \cite{hopmarszig}]\label{jjjnbbbnajkhypregdirecece}
Let $\mathcal{G}=(V,\mathcal{E})$ be a hypergraph, $h, k, \alpha, \beta \in \mathbb Z_+$ and $\ell, \ell': \mathbb{Z}_k \rightarrow  \mathbb{Z}_+$ such that~\eqref{totnecessary} and~\eqref{indivnecessary} hold. There exists an $h$-regular $(\ell, \ell')$-bordered $(\alpha, \beta)$-limited packing of $k$ rooted hyperforests in $\mathcal{G}$ if and only if \eqref{hValpha} holds and
\begin{alignat}{3}
	\beta - \ell(\mathbb{Z}_k) + \sum_{i=1}^k\min\{|\mathcal{P}|,\ell(i)\} + e_\mathcal{E}(\mathcal{P}) 	&	\ge h|\mathcal{P}| 	\qquad 	&&	\text{for every partition } \mathcal{P} \text{ of } V, \label{konfeiobfiueguye1mreg}\\
	\sum_{i=1}^k\min\{|\mathcal{P}|,\ell'(i)\} + e_\mathcal{E}(\mathcal{P}) 	&	\ge h|\mathcal{P}| 
	\qquad 	&&	\text{for every partition } \mathcal{P} \text{ of } V. \label{konfeiobfiueguye2mreg}
	\end{alignat}
\end{thm}

We mention that the natural extension of Theorems \ref{jjjnbbbnajkhypregdir} and \ref{jjjnbbbnajkhypregdirecece} to mixed hypergraphs does not hold, see  \cite{hopmarszig}.
\medskip

 We need to present the following result in order to deduce Corollary \ref{jfxghvbjkmix} of it. 

\begin{thm}[Fortier et al. \cite{FKLSzT}]\label{iuviumixhyp}
Let $\mathcal{F}=(V,\mathcal{E}\cup\mathcal{A})$  be a mixed hypergraph and ${\cal S}=\{S_1,\dots,S_k\}$ a family of subsets in $V$. There exists a packing of spanning mixed  hyperbranchings with root sets $S_1,\dots,S_k$  in $\mathcal{F}$ if and only if 
	\begin{eqnarray*}  
		e_{\mathcal{E}\cup\mathcal{A}}(\mathcal{P}) +\sum_{X\in\mathcal{P}}|\mathcal{S}_{X}|	&	\geq 		& 	
		k |\mathcal{P}|	\hskip .44truecm\text{ for every subpartition } \mathcal{P} \text{ of } V.
	\end{eqnarray*}
\end{thm}

If $\mathcal{F}$ is digraph, then Theorem \ref{iuviumixhyp} reduces to Theorem \ref{edmondsbranchings}.
 \medskip

The following result for digraphs was observed by Frank, we present it for dypergraphs to apply it later in the  proof of Theorem \ref{iycivsxcomgen}.
  
\begin{coro}\label{jfxghvbjkmix}
Let $\mathcal{F}=(V+s,\mathcal{E}\cup\mathcal{A})$  be a mixed hypergraph such that only arcs leave $s$ and let $F$ be a set of arcs in $\mathcal{A}$ leaving $s.$ There exists a packing of   $|F|$ spanning mixed  $s$-hyperarborescences in $\mathcal{F}$ each containing  an arc of $F$ if and only if 
	\begin{eqnarray*}
		e_{\mathcal{E}\cup\mathcal{A}}(\mathcal{P})	&	\geq	 	&	|F||\mathcal{P}| 
			\hskip .44truecm\text{ for every  subpartition } \mathcal{P} \text{ of } V.
	\end{eqnarray*}
\end{coro}

Corollary \ref{jfxghvbjkmix} follows from Theorem \ref{iuviumixhyp} by applying it for $\mathcal{S}=\{\{s,v\}: sv\in F\}.$

\section{Augmentation problems}

This section contains the augmentation versions of the results presented in Section \ref{Packingproblems}.
\medskip

The first results on arborescence packing augmentation   appeared in Frank \cite{book}.

\begin{thm}[Frank \cite{book}]\label{edmondsarborescencesaug} 
Let $D=(V,A)$ be a digraph, $k,\gamma\in \mathbb{Z}_+,$ and $S=\{s_1,\dots,s_k\}$ a multiset of vertices in $V$. 
We can add $\gamma$ arcs to  $D$ to have a packing of 

(a) $k$ spanning arborescences with roots $s_1,\dots,s_k$  if and only if  
	\begin{eqnarray*}
		 \gamma+e_A(\mathcal{P})+\sum_{X\in\mathcal{P}}|S_X|	&	\ge 	&	 k|\mathcal{P}|
		 \hskip .5truecm \text{ for every subpartition $\mathcal{P}$ of $V.$}
	\end{eqnarray*}

(b) $k$ spanning arborescences  if and only if 
	\begin{eqnarray*} 
		\gamma+e_A({\cal P})+k  &	\geq  &	k|{\cal P}| \hskip 1.6truecm \text{ for every subpartition $\mathcal{P}$ of $V.$}
	\end{eqnarray*}
\end{thm}

If $\gamma=0$, then Theorems \ref{edmondsarborescencesaug}(a) and (b) reduce to Theorems \ref{edmondsarborescencesmulti} and \ref{frankarborescences}.

\subsection{Augmentation to have a packing of mixed hyperarborescences}

We now provide the first main result of the paper which is a slight extension of the augmentation version of Theorem \ref{bjhkcgckh}. 

\begin{thm}\label{iycivsxcomgen} 
Let $\mathcal{F}=(V,\mathcal{E}\cup\mathcal{A})$  be a mixed hypergraph, $f,g,f',g': V\rightarrow \mathbb Z_+$ functions, $h,\alpha,\beta,q,q'\in \mathbb{Z}_+,$ {$S$} a multiset of vertices in $V$, and {${\sf M}$} $=(S,r_{\sf M})$ a matroid. 
We can add an $(f',g')$-indegree-bounded  $(q,q')$-size-limited arc set $F$ to $\mathcal{F}$ to have an $h$-regular ${\sf M}$-independent-rooted $(f,g)$-bounded $(\alpha, \beta)$-limited packing of mixed hyperarborescences that contains $F$ if and only if \eqref{fghu} and \eqref{ellll}  hold and  for all $v\in V$, $X, Z\subseteq V$ and subpartition $\mathcal{P}$ of $Z$,
\begin{eqnarray}
	f'(v)					&	\le 	& 	g'(v),		\label{rq}\\
	q					&	\le 	& 	q',		\label{mmmm}\\
	\max\{f'(Z),q-g'(\overline Z)\}+\max\{f(Z),\max\{h,\alpha\}-r_{\sf M}(S_{\overline X})+ f(Z-X)-g_h(X-Z)\}	&\le	&	h|Z|,						\label{utrdciyvouicomgen}\\
	e_{{\cal E}\cup{\cal A}}({\cal P})+\min\{g'(Z),q'-f'(\overline Z)\}+\min\{\beta-f(\overline Z),r_{\sf M}(S_{X})-f(X-Z)+g_h(Z-X)\}	
	&	\ge 	& 	h|\mathcal{P}|.\label{uhvigcftfcomgen}\hskip .8truecm
\end{eqnarray}
\end{thm}

\begin{proof}
The theory of generalized polymatroids provides the tools to be applied in the proof. The discovery of the submodularity of the function $e_{\mathcal{E}\cup \mathcal{A}}$ on subpartitions  makes the application of generalized polymatroids  possible.

\begin{cl}\label{submodeep}
Let  $\mathcal{P}_1$ and $\mathcal{P}_2$ be subpartitions of $V$ and $X \in \mathcal{E}\cup \mathcal{A}$.
	\begin{itemize}
		\item[(a)] If $X$ enters an element of  $\mathcal{P}_1 \sqcap \mathcal{P}_2$ or an element of 
	$\mathcal{P}_1 \sqcup \mathcal{P}_2$, then it enters an element of  $\mathcal{P}_1$ or an element of $\mathcal{P}_2$.
		\item[(b)] If $X$  enters an element of  $\mathcal{P}_1 \sqcap \mathcal{P}_2$ and an element of 
				$\mathcal{P}_1 \sqcup \mathcal{P}_2$, then it enters an element of  
				$\mathcal{P}_1$ and an element of $\mathcal{P}_2$. 
		\item[(c)] Consequently, 
			\begin{equation}\label{submodeepeq}
				e_{\mathcal{E}\cup \mathcal{A}}(\mathcal{P}_1) + 
				e_{\mathcal{E}\cup \mathcal{A}}(\mathcal{P}_2) \ge 
				e_{\mathcal{E}\cup \mathcal{A}}(\mathcal{P}_1 \sqcap \mathcal{P}_2) + 
				e_{\mathcal{E}\cup \mathcal{A}}(\mathcal{P}_1 \sqcup \mathcal{P}_2).
			\end{equation}
	\end{itemize}
\end{cl}

\begin{proof}
We prove (a) and (b) together. 
\medskip

First suppose that $X\in \mathcal{E}.$ 
If $X$ enters a subset $Y$ of $V$, then there exist $x\in X \cap Y$  and $z\in X-Y$. 

Suppose that $Y \in \mathcal{P}_1 \sqcup \mathcal{P}_2$. Since $x\in Y\subseteq (\cup\mathcal{P}_1)\cup(\cup\mathcal{P}_2)$, there exist $U_1 \in \mathcal{P}_1$ or $U_2 \in \mathcal{P}_2$ containing $x$, say $U_1 \in \mathcal{P}_1$. By \eqref{kvkjvk2} and $x\in Y \cap U_1$, we have $U_1 \subseteq Y$. Since $x\in U_1\cap X$ and $z\in X-Y\subseteq X-U_1,$ we get that $X$ enters $U_1$.

Suppose now that  $Y \in \mathcal{P}_1 \sqcap \mathcal{P}_2.$ Since $x\in Y\subseteq (\cup\mathcal{P}_1)\cap(\cup\mathcal{P}_2)$, there exist $U'_1 \in \mathcal{P}_1$ and $U'_2 \in \mathcal{P}_2$ containing $x$. By \eqref{kvkjvk1} and $x\in U'_1 \cap U'_2\cap Y$, we have $U'_1 \cap U'_2 \subseteq Y$. Since $x\in U'_1 \cap U'_2\cap X$ and $z\in X-Y\subseteq X-(U'_1 \cap U'_2),$ we get that $X$ enters $U'_1 \cap U'_2$. Hence, $X$ enters either $U'_1$ or $U'_2$. Thus, the proof of (a) is complete.

Further, if $X$ enters say  $U'_1$  but not $U'_2$, then, since $x\in X\cap U'_2,$ we have $X\subseteq U'_2.$ Then, by \eqref{kvkjvk2}, $X$ does not enter any element of $\mathcal{P}_1 \sqcup \mathcal{P}_2.$ Hence, the proof of (b) is complete.
\medskip

The case for $X\in \mathcal{A}$ is the same except that $x$ is the head of $X$ and  $z$ is a tail of $X$ not  in $Y$.
\medskip

It is clear that (c)  follows from (a) and (b).
\end{proof}

We introduce a set function $\hat p$: for all $Z\subseteq V,$ 
$$\bm{\hat p}(Z)=\max\{h|\mathcal{P}|-e_{\mathcal{E}\cup\mathcal{A}}({\cal P}): \mathcal{P} \text{ subpartition of } Z\}.$$

\begin{cl}\label{psup}
The function $\hat p$ is supermodular.
\end{cl}

\begin{proof}
Let $X_1,X_2\subseteq V.$ For $i=1,2,$ let ${\cal P}_i$ be the subpartition of $X_i$ such that $\hat p(X_i)=h|\mathcal{P}_i|-e_{\mathcal{E}\cup\mathcal{A}}({\cal P}_i).$ Then, by \eqref{kvkjvk3}, \eqref{submodeepeq}, $\mathcal{P}_1\sqcap \mathcal{P}_2$ and $\mathcal{P}_1\sqcup \mathcal{P}_2$ are subpartitions of $X_1\cap X_2$ and $X_1\cup X_2$, we have 
\begin{eqnarray*}
	\hat p(X_1)+\hat p(X_2)	&	=	&	h|{\cal P}_1|-e_{{\cal E}\cup{\cal A}}({\cal P}_1)+h|{\cal P}_2|-e_{{\cal E}\cup{\cal A}}({\cal P}_2)\\
						&	\le 	&	h|{\cal P}_1\sqcap {\cal P}_2|-e_{{\cal E}\cup{\cal A}}({\cal P}_1\sqcap {\cal P}_2)+h|{\cal P}_1\sqcup {\cal P}_2|-e_{{\cal E}\cup{\cal A}}({\cal P}_1\sqcup {\cal P}_2)\\
						&	\le 	&	\hat p(X_1\cap X_2)+\hat p(X_1\cup X_2),
\end{eqnarray*}
 and the claim follows.
\end{proof}

We naturally consider the following $g$-polymatroids, where {\boldmath$b_{\sf M}$}$(X)=r_{\sf M}(S_X)$ for every $X\subseteq V,$
which  are well-defined because $b_{\sf M}$ is submodular and, by Claim \ref{psup}, $\hat p$ is supermodular.
\medskip

\hskip 1truecm {\boldmath$Q_1$} $=T(f',g')\cap K(q,q'),$ 

\hskip 1truecm {\boldmath$Q_2$} $=Q(-\infty_0,b_{\sf M})\cap K(\max\{h,\alpha\},\beta)\cap T(f,g_h)$, 

\hskip 1truecm {\boldmath$Q_3$} $=Q(\hat p,\infty_0)\cap T(0,h)$,    

\hskip 1truecm {\boldmath$Q_4$} $=(Q_1+Q_2)\cap Q_3.$

\begin{Lemma}\label{gpl2} The following hold.
	\begin{itemize}
		\item[(a)] An integral vector $m$ is in $Q_4$ if and only if there exist an  $(f',g')$-indegree-bounded $(q,q')$-size-limited  new arc set $F$  and an $h$-regular ${\sf M}$-independent-rooted  $(f,g)$-bounded $(\alpha,\beta)$-limited  packing of mixed hyperarborescences in $\mathcal{F}+F$ that contains $F$ with root set $R$ such that   $m(v)=d_F^-(v)+|R_v|$ for every $v\in V.$
		\item[(b)] $Q_4\neq\emptyset$ if and only if \eqref{fghu}, \eqref{rq}, \eqref{mmmm}, \eqref{ellll},   \eqref{utrdciyvouicomgen}, and \eqref{uhvigcftfcomgen} hold.
	\end{itemize}
\end{Lemma}

\begin{proof}
(a) To prove the {\bf necessity}, suppose  that there exist an $(f',g')$-indegree-bounded  $(q,q')$-size-limited new arc set {\boldmath$F$}  and an $h$-regular ${\sf M}$-independent-rooted $(f,g)$-bounded $(\alpha,\beta)$-limited packing {\boldmath$\mathcal{B}$} of mixed hyperarborescences in $\mathcal{F}+F$ that contains $F.$ Let {\boldmath$R$} be the root set of $\mathcal{B}$. Then we have the following.
\begin{alignat}{6}
	|R_X|	&	\le 	r_{\sf M}(S_X) 	&& =  b_{\sf M}(X)	\qquad &&\text{for every }X\subseteq V,	\label{cnmr1}\\
\max\{h,\alpha\}	&	\le 		|R|		&&	\le 	\beta,	\label{cn1}\\
		f(v)	&	\le 		|R_v|		&&	\le 	g(v) 		\qquad&&\text{for every } v\in V,\label{cn2}\\
		q	&	\le 		|F|		&&	\le 	q',		\label{cn3}\\
		f'(v)	&	\le 		d_F^-(v)	&&	\le 	g'(v)		\qquad&&\text{for every }  v\in V,\label{cn4}\\
		0	&	\le 	|R_v|+d_F^-(v)	&&	\le 	h,		\qquad&&\text{for every }  v\in V.\label{cn5}
\end{alignat}
Let  {\boldmath$m_1$}$(v)=d_F^-(v)$, {\boldmath$m_2$}$(v)=|R_v|$ for every $v\in V$ and {\boldmath$m$} $=m_1+m_2.$ Then, by \eqref{cnmr1}, \eqref{cn1}, \eqref{cn2}, and \eqref{cn5}, we have $m_2\in \mathbb{Z}\cap Q_2$, by \eqref{cn3} and \eqref{cn4}, we have $m_1\in \mathbb{Z}\cap Q_1,$ and, by \eqref{cn5}, we have $m\in T(0,h).$ Since $m=m_1+m_2,$ we have  $m\in \mathbb{Z}\cap (Q_1+Q_2)$. It remains  to show that $m\in Q(\hat p,\infty_0).$ 
By Theorem \ref{sibevhz2} applied for $f(v)=g(v)=|R_v|$ for every $v\in V$ and $\alpha=\beta=|R|,$ we get that, for every $\emptyset\neq X\subseteq V$, we have 
\begin{eqnarray*}
	\hat p(X)	&	=	&	\max\{h|{\cal P}|-e_{{\cal E}\cup{\cal A}}({\cal P}): {\cal P} \text{ subpartition of } X\}\\
	&	\le	&	\max\{h|{\cal P}|-e_{{\cal E}\cup{\cal A}\cup F}({\cal P})+\sum_{v\in X}d_F^-(v): {\cal P}\text{ subpartition of } X\}\\
	&	\le 	&	|R_{\cup{\cal P}}|+\sum_{v\in X}d_F^-(v)\le |R_X|+\sum_{v\in X}d_F^-(v)
		=		m_2(X)+m_1(X)
		=		m(X),
\end{eqnarray*}
so $m\in Q(\hat p,\infty_0).$ Thus, $m\in \mathbb{Z}\cap Q_4.$
\medskip
 
To prove the {\bf sufficiency}, now  suppose  that {\boldmath$m$} $\in \mathbb{Z}\cap Q_4.$  By $m\in \mathbb{Z}\cap (Q_1+Q_2)$ and Theorem \ref{gpm}.\ref{sum}, there exist {\boldmath$m_1$} $\in \mathbb{Z}\cap Q_1$ and {\boldmath$m_2$} $\in \mathbb{Z}\cap Q_2$ such that $m=m_1+m_2.$ Then, by $f,f'\ge 0$, we have $m_1,m_2\in \mathbb{Z}_+.$ Let the mixed hypergraph {\boldmath$\mathcal{F}_1$} be obtained from $\mathcal{F}$ by adding a new vertex {\boldmath$s$} and a new arc set {\boldmath$F_3$} $=F_1\cup F_2$, where for $i=1,2,$ {\boldmath$F_i$} is a new arc set  containing $m_i(v)$ arcs from $s$ to every  $v\in V.$ Then $d_{F_1}^-=m_1, d_{F_2}^-=m_2$ and $d_{F_3}^-=m.$

We first show that  $\mathcal{F}_1$ contains a packing of $h$ spanning mixed $s$-hyperarborescences, each containing at least one arc of $F_2$, and their union containing $F_3.$ By $d_{F_2}^-=m_2\in K(\max\{h,\alpha\},\beta),$   we have $|F_2|=m_2(V)\ge\ h.$ Thus, there exists a subset {\boldmath$F_4$} of $F_2$ of size $h.$ By $m\in Q_3, $ we have $e_{F_3}({\cal P})=d_{F_3}^-(\cup{\cal P})=m(\cup{\cal P})\ge \hat p(\cup{\cal P})\ge h|{\cal P}|-e_{{\cal E}\cup{\cal A}}({\cal P})=|F_4||{\cal P}|-e_{{\cal E}\cup{\cal A}}({\cal P})$ for every  subpartition $\mathcal{P}$ of $V.$ Then, by Corollary \ref{jfxghvbjkmix},  there exists a packing {\boldmath$\mathcal{B}^1$} of $h$ spanning mixed $s$-hyperarborescences in $\mathcal{F}_1$ each containing an arc of $F_4.$ By $m\in T(0,h)$, we have $d_{F_3}^-(v)=m(v)\le h$ for every $v\in V.$ We may hence modify $\mathcal{B}^1$ to obtain a packing {\boldmath$\mathcal{B}^2$} of $h$ spanning mixed $s$-hyperarborescences in $\mathcal{F}_1$ each containing an arc of $F_4$ and their union containing all the arcs  in $F_3.$ 

We now modify each spanning mixed $s$-hyperarborescence $B^2_i$ in $\mathcal{B}^2$ by deleting $s$ and adding an arc $vu$ for every arc $su\in F_1$ contained in $B_i^2$ where $sv$ is the unique arc of $F_4$ contained in  $B_i^2$. Let {\boldmath$F$} be the set of new arcs. This way we obtained an $h$-regular packing {\boldmath$\mathcal{B}^3$} of mixed hyperarborescences in $\mathcal{F}+F$ that contains $F.$ Note that the root set $R^3$ of the packing $\mathcal{B}^3$ satisfies $|R^3_v|=d_{F_2}^-(v)=m_2(v)$ for every $v\in V.$ Then, since $m_2\in T(f,g_h),$ $\mathcal{B}^3$ is $(f,g)$-bounded. By $m_2\in K(\max\{h,\alpha\},\beta),$   $\mathcal{B}^3$ is $(\alpha,\beta)$-limited. Further, by $m_2\in Q(-\infty_0,b_{\sf M})$ and Theorem 13.1.2  in \cite {book}, there exists an independent set {\boldmath$S^*$} in {\sf M} such that  for every $v\in V$, $|S^*_v|=m_2(v)=|R^3_v|$ and hence the $m_2(v)$ mixed $v$-hyperarborescences in $\mathcal{B}^3$ can be rooted at $S^*_v$. It follows that the packing $\mathcal{B}^3$ is {\sf M}-independent-rooted. Since $d_F^-(v)=d_{F_1}^-(v)=m_1(v)$ for every $v\in V$ and $m_1\in T(f',g')\cap K(q,q'),$ $F$ is $(f',g')$-indegree-bounded and $(q,q')$-size-limited, so the proof of (a) is completed.
\medskip

(b) By Theorem \ref{gpm}.\ref{TQ}  and \ref{gpm}.\ref{QcapK}, $Q_1=T(f',g')\cap K(q,q')\neq\emptyset$ if and only if $f'\le g'$ (which is \eqref{rq}), $q\le q'$ (which is \eqref{mmmm}), $g'(V)\ge q$  (which is implied by  \eqref{utrdciyvouicomgen} applied for $Z=\emptyset$) and $f'(V)\le q'$ (which is implied by  \eqref{uhvigcftfcomgen} applied for $X=Z={\cal P}=\emptyset$). If $Q_1\neq\emptyset,$ then, by Theorem \ref{gpm}.\ref{QcapK},  $Q_1=Q(p_1,b_1)$ where, for every  $X\subseteq V,$
\begin{eqnarray}
	\text{{\boldmath$p_1$}}(X)	&	=	&	\max\{f'(X), q-g'(\overline X)\},\label{txyuciy}\\
	\text{{\boldmath$b_1$}}(X)	&	=	&	\min\{g'(X), q'-f'(\overline X)\}.\label{ohuifytu}
\end{eqnarray}
By Theorem  \ref{gpm}.\ref{QcapK}, $Q'=Q(-\infty_0,b_{\sf M})\cap K(\max\{h,\alpha\},\beta)\neq\emptyset$ if and only if $\max\{h,\alpha\}\le b_{\sf M}(V)=r_{\sf M}(S)$ which is implied by  \eqref{utrdciyvouicomgen} applied for $X=\emptyset=Z.$ If $Q'\neq\emptyset,$ then by Theorem \ref{gpm}.\ref{QcapK}, $Q'=Q(p',b')$  where, for every  $X\subseteq V,$
\begin{eqnarray}
	\text{{\boldmath$p'$}}(X)	&	=	&	\max\{-\infty_0(X),\max\{h,\alpha\}-b_{\sf M}(\overline X)\}
		=\max\{h,\alpha\}-b_{\sf M}(\overline X) \text{ if } X\neq\emptyset \text{ and } 0\text{  if } X=\emptyset,\label{jhvjhvhj}\hskip .6truecm\\
	\text{{\boldmath$b'$}}(X)	&	=	&	\min\{b_{\sf M}(X),\beta+\infty_0(\overline X)\}=b_{\sf M}(X) \text{ if } X\neq V \text{ and } \min\{b_{\sf M}(V),\beta\} \text{  if } X=V.\label{jgxgfx}
\end{eqnarray}
By Theorem \ref{gpm}.\ref{QcapT}, $Q_2=Q'\cap T(f,g_h)=Q(p',b')\cap T(f,g_h)\neq\emptyset$ if and only if $\max\{h,\alpha\}\le b_{\sf M}(V)=r_{\sf M}(S)$, $f\le g_h$ which is \eqref{fghu}, $f(X)\le b'(X)$  for every $X\subseteq V$ which, by \eqref{jgxgfx}, is \eqref{uhvigcftfcomgen} applied for $Z=\emptyset=\mathcal{P}$ and $p'(X)\le g_h(X)$  for every $X\subseteq V$ which, by \eqref{jhvjhvhj},  is \eqref{utrdciyvouicomgen} applied for $Z=\emptyset$. If $Q_2\neq\emptyset$, then, by Theorem \ref{gpm}.\ref{QcapT}, $Q_2=Q(p_2,b_2)$ where, for every  $Z\subseteq V,$
\begin{eqnarray}
	\text{{\boldmath$p_2$}}(Z)	&	=	&	\max\{p'(X)-g_h(X-Z)+f(Z-X):X\subseteq V\},	\label{uyfohjvlj}\\
	\text{{\boldmath$b_2$}}(Z)	&	=	&	\min\ \{b'(X)-f(X-Z)+g_h(Z-X):X\subseteq V\}.	\label{oufiyfyu}
\end{eqnarray}
By Theorem \ref{gpm}.\ref{sum}, $Q_1+Q_2=Q(p_+,b_+)$ where, for every  $Z\subseteq V,$
\begin{eqnarray}
	\text{{\boldmath$p_+$}}(Z)		&	=	&	p_1(Z)+p_2(Z),		\label{yuyguyu}\\
	\text{{\boldmath$b_+$}}(Z)		&	=	&	b_1(Z)+b_2(Z).		\label{ouyfiufou}
\end{eqnarray}
As the vector $(h,\dots, h)\in Q(\hat p,\infty_0)\cap T(0,h)=Q_3,$  $Q_3\neq\emptyset.$ Then, by Theorem \ref{gpm}.\ref{QcapT}, $Q_3=Q(p_3,b_3)$ where, for every  $X\subseteq V,$
\begin{eqnarray}
	\text{{\boldmath$p_3$}}(X)	&	=	&	\hat p(X),\label{bfkjbf}\\
	\text{{\boldmath$b_3$}}(X)	&	=	&	h|X|.\label{lbkjebjlefjl}
\end{eqnarray}
By Theorem \ref{gpm}.\ref{QcapQ}, $Q_4\neq\emptyset$ if and only if the above conditions hold and $p_3\le b_+$ (which, by   \eqref{bfkjbf}, \eqref{ouyfiufou}, \eqref{ohuifytu},  \eqref{oufiyfyu}, is equivalent to \eqref{uhvigcftfcomgen}) and $p_+\le b_3$ (which, by \eqref{yuyguyu},   \eqref{txyuciy}, \eqref{uyfohjvlj},  and \eqref{lbkjebjlefjl}, is equivalent to \eqref{utrdciyvouicomgen}). By summarizing the above arguments we may conclude  that (b) holds.
\end{proof}

By Lemma \ref{gpl2}(a), there exists an $(f',g')$-indegree-bounded $(q,q')$-size-limited  new arc set $F$ such that $\mathcal{F}+F$ has an $(f,g)$-bounded ${\sf M}$-independent-rooted $(\alpha,\beta)$-limited $h$-regular packing of mixed hyperarborescences containing $F$ if and only if  $m=d_F^-+|R.|\in \mathbb{Z}\cap Q_4$. By Theorem \ref{gpm}.\ref{QcapQ}, this is equivalent to $Q_4\neq\emptyset$. This holds, by Lemma \ref{gpl2}(b), if and only if \eqref{fghu}, \eqref{rq}, \eqref{mmmm}, \eqref{ellll},   \eqref{utrdciyvouicomgen}, and \eqref{uhvigcftfcomgen} hold. This completes the proof of Theorem \ref{iycivsxcomgen}.
\end{proof}

If $f'=g'=q=q'=0$, then Theorem \ref{iycivsxcomgen} reduces to Theorem \ref{bjhkcgckh}. If $f(v)=|S_v|=g(v), f'(v)=0, g'(v)=\infty$ for every $v\in V$, $h=\alpha=\beta=k$, $q=0,q'=\gamma,$ {\sf M} is the free matroid on $h\times V$, then Theorem \ref{iycivsxcomgen} reduces to Theorem \ref{edmondsarborescencesaug}(a). If $f(v)=0, g(v)=k, f'(v)=0, g'(v)=\infty$ for every $v\in V$, $h=\alpha=\beta=k$, $q=0,q'=\gamma,$ {\sf M} is the free matroid on $h\times V$, then Theorem \ref{iycivsxcomgen} reduces to Theorem \ref{edmondsarborescencesaug}(b). If $q=0,q'=\gamma, f'(v)=0, g'(v)=\infty$ for every $v\in V$, then  Theorem \ref{iycivsxcomgen} implies Corollary \ref{iuiggiiu}, the augmentation version of Theorem \ref{bjhkcgckh}.

\begin{coro}\label{iuiggiiu} 
Let $\mathcal{F}=(V,\mathcal{E}\cup\mathcal{A})$  be a mixed hypergraph, $f,g: V\rightarrow \mathbb Z_+$ functions, $h,\alpha,\beta,\gamma\in \mathbb{Z}_+,$ {$S$} a multiset of vertices in $V$, and {${\sf M}$} $=(S,r_{\sf M})$ a matroid. 
We can add an arc set $F$ of size at most $\gamma$ to $\mathcal{F}$ to have an $h$-regular ${\sf M}$-independent-rooted $(f,g)$-bounded $(\alpha, \beta)$-limited packing of mixed hyperarborescences that contains $F$ if and only if \eqref{fghu}, \eqref{ellll}, and \eqref{utrdciyvouicomgenj} hold and  for all $X, Z\subseteq V$ and subpartition $\mathcal{P}$ of $Z$,
\begin{eqnarray}
		\beta			&	\ge 	& 	f(V),\\
	r_{\sf M}(S_{X})		&	\ge 	& 	f(X),\\
	\gamma+e_{{\cal E}\cup{\cal A}}({\cal P})+\min\{\beta-f(\overline Z),r_{\sf M}(S_{X})-f(X-Z)+g_h(Z-X)\}	
					&	\ge 	& 	h|\mathcal{P}|.\label{uhvigcftfcomgen}\hskip .8truecm
\end{eqnarray}
\end{coro}

If $\gamma=0,$ then Corollary \ref{iuiggiiu} reduces to Theorem \ref{bjhkcgckh}.

\subsection{Augmentation to have a packing of hyperbranchings}

We now provide the second main result of the paper which is  the augmentation version of Theorem  \ref{jjjnbbbnajkhypregdir}.

\begin{thm}\label{kjvkch}
Let $\mathcal{D} = (V, \mathcal{A})$ be a dypergraph, $h, k, \alpha, \beta, \gamma \in \mathbb Z_{+},$ and $\ell, \ell': \mathbb{Z}_k \rightarrow  \mathbb{Z}_+$ such that~\eqref{totnecessary} and~\eqref{indivnecessary} hold. We can add at most $\gamma$ arcs to $\mathcal{D}$ to have an $h$-regular $(\ell, \ell')$-bordered $(\alpha, \beta)$-limited packing of $k$  hyperbranchings 
if and only if     \eqref{hValpha} holds and 
\begin{eqnarray}  
	k	&	\ge	& 	h,	\label{kgeqh}\\
	\gamma+\beta - \ell(\mathbb{Z}_k) + \sum_{i=1}^k\min\{|\mathcal{P}|,\ell(i)\} + e_\mathcal{A}(\mathcal{P}) 		&	\ge	& 	h|\mathcal{P}| \qquad \text{for every subpartition } \mathcal{P} \text{ of } V, 	\label{konfeiobfiueguye1mregdirgam}\\
	\gamma+\sum_{i=1}^k\min\{|\mathcal{P}|,\ell'(i)\} + e_\mathcal{A}(\mathcal{P}) 		&	\ge	& 	h|\mathcal{P}| 
	\qquad \text{for every subpartition } \mathcal{P} \text{ of } V. 	\label{konfeiobfiueguye2mregdirgam}
\end{eqnarray}
\end{thm}

\begin{proof}
The {\bf necessity} easily follows from Theorem \ref{jjjnbbbnajkhypregdir} and the fact that $e_{\mathcal{A}\cup F}(\mathcal{P})\le e_\mathcal{A}(\mathcal{P})+\gamma$, where $F$ is the new arc set added to $\mathcal{D}.$ To prove the {\bf sufficiency}, suppose that \eqref{totnecessary}, \eqref{indivnecessary}, \eqref{hValpha}, \eqref{kgeqh}, \eqref{konfeiobfiueguye1mregdirgam}, and \eqref{konfeiobfiueguye2mregdirgam} hold. 
We need the following technical augmentation lemma.

\begin{Lemma}\label{efegegrg} 
Let $h, k,  n,\alpha,\beta,\gamma\in \mathbb Z_{+}, \ell,\ell': \mathbb Z_k\rightarrow  \mathbb Z_+$, and $e: \{1,\dots, n\}\rightarrow  \mathbb Z_{+}$.
 There exist $\hat\alpha,\hat\beta\in \mathbb Z_{+},$ and  $\hat\ell,\hat\ell': \mathbb Z_k\rightarrow  \mathbb Z_+$ such that 
   \begin{alignat}{3}
        	n \ge \hat{\ell}'(i) &\ge \hat{\ell}(i)  &&\text{for every } 1 \le i \le k, 							\label{newindcond}\\
        	0 \le \hat{\ell}'(i) - \ell'(i) &\le \hat{\ell}(i) - \ell(i)  &&\text{for every } 1 \le i \le k,					\label{mvlkl} \\
        	hn &\ge \hat\alpha, 																\label{alphareg}\\
        	\hat{\ell}'(\mathbb Z_k)\ge\hat\beta \ge \hat\alpha&\ge\hat{\ell}(\mathbb Z_k), 					\label{newtotcond}\\
        	\hat\alpha - \alpha = \hat\beta - \beta = \hat{\ell}(\mathbb Z_k) &- \ell(\mathbb Z_k) \le {\gamma},	\label{uhviuycyi} \\
      	\hat\beta+\sum\nolimits_{i = 1}^k\min\{p-\hat{\ell}(i),0\}  &\ge hp-e(p)&&\text{for every }1\le p\le n, 	\label{newcondell}\\
	\hat{\ell}'(\mathbb Z_k)+ \sum\nolimits_{i = 1}^k\min\{p-\hat{\ell}'(i),0\}&\ge hp-e(p)\qquad&&\text{for every }1\le p\le n,		\label{newcondell'}
    \end{alignat}
    if and only if 
\begin{alignat}{3}
        n \ge \ell'(i) &\ge \ell(i) \qquad &&\text{for every } 1 \le i \le k, \label{indnecess}\\
        hn &\ge \alpha,\label{miuviclkb} \\
	\ell'(\mathbb Z_k) \ge \beta \ge \alpha &\ge\ell(\mathbb Z_k), \label{totnecess}\\
	kp&\ge hp-e(p) \qquad &&\text{for every } 1 \le p \le n,\label{jkbfytdydyi}\\
	\gamma+\beta+\sum\nolimits_{i = 1}^k\min\{p-\ell(i),0\}&\ge hp-e(p)\qquad&&\text{for every }1\le p\le n,\label{newcondellrgfr}\\
	\gamma+\ell'(\mathbb Z_k)+\sum\nolimits_{i = 1}^k\min\{p-\ell'(i),0\}&\ge hp-e(p)\qquad&&\text{for every }1\le p\le n.			\label{newcondell'rrfgg}
\end{alignat}
\end{Lemma}

\begin{proof}
We first show the {\bf necessity}. Suppose that \eqref{newindcond}--\eqref{newcondell'} hold for some $\hat{\alpha},\hat{\beta},\hat\ell$, and $\hat\ell'$. By \eqref{newindcond} and \eqref{mvlkl}, we have, for every $1\le i\le k,$ $\ell(i)\le\hat\ell'(i)-\hat\ell(i)+\ell(i)\le\ell'(i)\le\hat\ell'(i)\le n,$ so \eqref{indnecess} holds. By \eqref{alphareg}, \eqref{uhviuycyi}, and \eqref{mvlkl}, we have $hn-\alpha\ge\hat\alpha-\alpha=\hat\ell(\mathbb Z_k)-\ell(\mathbb Z_k)\ge 0,$ so \eqref{miuviclkb} holds. By  \eqref{mvlkl} and \eqref{uhviuycyi}, we have $\ell'(\mathbb Z_k)-\hat\ell'(\mathbb Z_k)\ge \beta-\hat\beta=\alpha-\hat\alpha=\ell(\mathbb Z_k)-\hat\ell(\mathbb Z_k).$ By adding this to \eqref{newtotcond}, we obtain \eqref{totnecess}. By  \eqref{newcondell'}, we have $kp\ge\sum\nolimits_{i = 1}^k\min\{p,\hat{\ell}'(i)\}\ge hp-e(p)$, for every $1\le p\le n$, so \eqref{jkbfytdydyi} holds. By \eqref{uhviuycyi}, \eqref{mvlkl}, and \eqref{newcondell}, we have $\gamma+\beta+\sum_{i = 1}^k\min\{p-\ell(i),0\}\ge \hat\beta+\sum_{i = 1}^k\min\{p-\hat\ell(i),0\}\ge hp-e(p)$ for every $1\le p\le n,$ so \eqref{newcondellrgfr} holds. By \eqref{uhviuycyi}, \eqref{mvlkl}, and \eqref{newcondell'}, we have $\gamma+{\ell}'(\mathbb Z_k)+\sum_{i = 1}^k\min\{p-\ell'(i),0\}\ge \hat{\ell}'(\mathbb Z_k)+\sum_{i = 1}^k\min\{p-\hat\ell'(i),0\}\ge hp-e(p)$ for every $1\le p\le n,$ so \eqref{newcondell'rrfgg} holds.
\medskip

To prove the {\bf sufficiency}, suppose that \eqref{indnecess}--\eqref{newcondell'rrfgg} hold. Let {\boldmath$\hat\gamma$} 
$\in\mathbb{Z}_+$ be minimum such that there exist  {\boldmath$\hat\alpha,\hat\beta$} $\in\mathbb{Z}_{+},$ and {\boldmath$\hat\ell,\hat\ell'$} $: \mathbb Z_k\rightarrow  \mathbb Z_+$ satisfying \eqref{newindcond}--\eqref{newtotcond} and 
   \begin{alignat}{3}
  \hat\alpha - \alpha = \hat\beta - \beta = \hat{\ell}(\mathbb Z_k) - \ell(\mathbb Z_k) &\le \gamma-\hat{\gamma},\label{uhviuycyihat} \\
      \hat{\gamma}+\hat\beta+\sum_{i = 1}^k\min\{p-\hat{\ell}(i),0\}  &\ge hp-e(p)&&\text{for every } 1\le p\le n, \label{newcondellhat}\\
\hat{\gamma}+\hat{\ell}'(\mathbb Z_k)+\sum_{i=1}^k\min\{p-\hat{\ell}'(i),0\}&\ge hp-e(p) \qquad &&\text{for every }1\le p\le n. \label{newcdell'hat}
    \end{alignat}
Since \eqref{indnecess}--\eqref{newcondell'rrfgg} hold for $\alpha,\beta,{\gamma},\ell$, and $\ell'$, it follows that $\hat\alpha,\hat\beta,\hat{\gamma},\hat\ell$, and $\hat\ell'$ exist. We show that $\hat{\gamma}=0$, and hence, by \eqref{uhviuycyihat}--\eqref{newcdell'hat}, it follows that  $\hat\alpha,\hat\beta,\hat\ell,$ and $\hat\ell'$ also satisfy \eqref{uhviuycyi}--\eqref{newcondell'} and we are done.
Suppose that $\hat{\gamma}>0.$ By the minimality of $\hat\gamma$, there exists a minimum {\boldmath$\hat p$} $\in\{1,\dots,n\}$ such that at least one of \eqref{newcondellhat} and \eqref{newcdell'hat} is tight for $\hat p.$ Let {\boldmath$X$} be the set of indices where $\hat p-\hat\ell(i)\le 0$ and {\boldmath$X'$}  the set of indices where $\hat p-\hat\ell'(i)\le 0.$ We first prove the following claim.
\begin{cl}\label{bonpop}
 The following hold.
\begin{itemize}
\item[(a)] $hn>\hat\alpha.$
\item[(b)] If \eqref{newcondellhat} is tight for $\hat p$, then $X\neq \mathbb Z_k.$
\item[(c)] If $X'=\mathbb Z_k,$ then \eqref{newcdell'hat} is strict for every $1\le p\le n$ and $\hat{\ell}'(\mathbb Z_k) >\hat\beta$.  
\end{itemize}
\end{cl}
\begin{proof}
{\bf (a)} If $\hat\alpha\ge hn,$ then, by \eqref{newtotcond}, we have 
	\begin{eqnarray}\label{mbivuvuy}
		\hat{\ell}'(\mathbb Z_k)\ge\hat\beta \ge hn.
	\end{eqnarray}
 
\noindent {\bf Case 1.} If \eqref{newcondellhat} is tight for $\hat p,$ then, by $e\ge 0$ and $\hat\gamma>0,$ we have 
	\begin{eqnarray}\label{mobuyfutd}
		h\hat p>\hat\beta+\sum_{i = 1}^k\min\{\hat p-\hat{\ell}(i),0\}=\hat\beta+|X|\hat p-\hat\ell(X).
	\end{eqnarray}
If $|X|\ge h,$ then, by \eqref{mobuyfutd}, $\hat\ell\ge 0$,  and \eqref{newtotcond},  we have  a contradiction: $h\hat p>\hat\beta+|X|\hat p-\hat\ell(X)\ge \hat\beta-\hat\ell(\mathbb Z_k)+h\hat p\ge h\hat p.$ If $|X|\le h,$ then, by \eqref{mobuyfutd}, \eqref{mbivuvuy},  $n\ge \hat p,$ and \eqref{indnecess},   we have  a contradiction: $h\hat p>\hat\beta+|X|\hat p-\hat\ell(X)\ge (h-|X|)n+|X|n+|X|\hat p-\hat\ell(X)\ge h\hat p+\sum_{i\in X}(n-\hat\ell(i))\ge h\hat p.$
\medskip

\noindent {\bf Case 2.} If \eqref{newcdell'hat} is tight for $\hat p,$ then, by $e\ge 0$ and $\hat\gamma>0,$ we have 
	\begin{eqnarray}\label{mkbhutr}
		h\hat p>\hat{\ell}'(\mathbb Z_k)+\sum_{i = 1}^k\min\{\hat p-\hat{\ell}'(i),0\}=\hat{\ell}'(\mathbb Z_k)+|X'|\hat p-\hat\ell'(X').
	\end{eqnarray}
 If $|X'|\ge h,$ then, by \eqref{mkbhutr}, and $\hat\ell'\ge 0$   we have  a contradiction: $h\hat p>\hat{\ell}'(\mathbb Z_k)+|X'|\hat p-\hat\ell'(X')\ge \hat{\ell}'(\mathbb Z_k)+h\hat p-\hat\ell'(\mathbb Z_k)\ge h\hat p.$ If $|X'|\le h,$ then, by \eqref{mkbhutr}, \eqref{mbivuvuy}, $n\ge \hat p,$ and \eqref{indnecess},   we have  a contradiction: $h\hat p>\hat{\ell}'(\mathbb Z_k)+|X'|\hat p-\hat\ell'(X')\ge (h-|X'|)n+|X'|n+|X'|\hat p-\hat\ell'(X')\ge h\hat p+\sum_{i\in X'}(n-\hat\ell'(i))\ge h\hat p.$
\medskip
		
\noindent {\bf (b)} If $X=\mathbb Z_k$, then, by $\hat{\gamma}>0$, \eqref{newtotcond},  \eqref{newcondellhat} is tight for $\hat p$, and \eqref{jkbfytdydyi}, we have a contradiction: $k\hat p<\hat{\gamma}+\hat\beta-\hat{\ell}(\mathbb Z_k)+k\hat p=\hat{\gamma}+\hat\beta+\sum_{i = 1}^k\min\{\hat p-\hat{\ell}(i),0\}=h\hat p-e(\hat p)\le k\hat p.$
\medskip

\noindent {\bf (c)} Suppose that $X'=\mathbb Z_k$. If $p<\hat p,$ then, by the minimality of $\hat p$, \eqref{newcdell'hat} is strict for $p$.

If $p=\hat p,$ then, by  \eqref{jkbfytdydyi}, $\hat{\gamma}>0$, and $X'=\mathbb Z_k$,  we have $h\hat p-e(\hat p)\le k\hat p<\hat{\gamma}+\hat{\ell}'(\mathbb Z_k)-\hat{\ell}'(\mathbb Z_k)+k\hat p=\hat{\gamma}+\hat{\ell}'(\mathbb Z_k)+\sum_{i = 1}^k\min\{\hat p-\hat{\ell}'(i),0\},$
so \eqref{newcdell'hat} is strict for $\hat p$. Note that in this case, by the definition of $\hat p$, \eqref{newcondellhat} is tight for $\hat p$ and hence 
	\begin{eqnarray}\label{ljvytusyui}
		\hat\beta<k\hat p-\sum_{i = 1}^k\min\{\hat p-\hat{\ell}(i),0\}=\sum_{i = 1}^k\max\{\hat p,\hat{\ell}(i)\}\le \sum_{i = 1}^k\hat{\ell}'(i)=\hat{\ell}'(\mathbb Z_k). 
	\end{eqnarray}

If $p\ge \hat p,$ then, for every $1\le i\le k,$ by $\max\{\hat p,\hat{\ell}(i)\}\le \hat{\ell}'(i)$, we have $\hat p\le \min\{p,\hat{\ell}'(i)\}-0,$ $\hat{\ell}(i)\le p-(p-\hat{\ell}(i))$ and $\hat{\ell}(i)\le \hat{\ell}'(i)-0$,  so $\max\{\hat p,\hat{\ell}(i)\}\le \min\{p,\hat{\ell}'(i)\}-\min\{p-\hat{\ell}(i),0\}.$ Hence, we have
	\begin{eqnarray}\label{jbmgm}
		\sum_{i = 1}^k\max\{\hat p,\hat{\ell}(i)\}\le \sum_{i = 1}^k\min\{p,\hat{\ell}'(i)\}-\sum_{i = 1}^k\min\{p-\hat{\ell}(i),0\}.
	\end{eqnarray}
By \eqref{ljvytusyui}, \eqref{jbmgm}, and since \eqref{newcondellhat} is tight for $\hat p$, we obtain that \eqref{newcdell'hat} is strict for $p$. This completes the proof of Claim \ref{bonpop}.
\end{proof}

We now distinguish two cases and we give a contradiction in both cases.
\medskip

\noindent {\bf Case I.} Suppose that $X'\neq \mathbb Z_k.$  Let $m\in \overline{X'}.$ Then, by \eqref{newindcond}, $\hat p>\hat\ell'(m)\ge \hat\ell(m).$ 
\medskip

Let $\tilde\gamma=\hat\gamma-1,$ $\tilde\alpha=\hat\alpha+1,$ $\tilde\beta=\hat\beta+1,$ $\tilde\ell(m)=\hat\ell(m)+1, \tilde\ell'(m)=\hat\ell'(m)+1$ and $\tilde\ell(i)=\hat\ell(i),\tilde\ell'(i)=\hat\ell'(i)$ for $i\in \mathbb Z_k-\{m\}.$ 
\medskip

By \eqref{newindcond} and $n\ge \hat p>\hat\ell'(m),$ we have $n\ge\tilde\ell'(i)\ge\tilde\ell(i)$ for every $1\le i\le k.$ By \eqref{mvlkl}, we have $0\le\tilde{\ell}'(i)-\ell'(i)\le\tilde{\ell}(i)-\ell(i)$ for every $1\le i\le k.$ By Claim \ref{bonpop}(a), we have $hn\ge\tilde\alpha.$ By \eqref{newtotcond}, we have $\tilde{\ell}'(\mathbb Z_k)\ge\tilde\beta\ge\tilde\alpha\ge\tilde{\ell}(\mathbb Z_k).$ By \eqref{uhviuycyi}, we have $\tilde\alpha-\alpha=\tilde\beta-\beta=\tilde{\ell}(\mathbb Z_k)-\ell(\mathbb Z_k)\le{\gamma}-\tilde\gamma.$ By $\tilde{\gamma}+\tilde\beta=\hat{\gamma}+\hat\beta$, $\sum_{i = 1}^k\min\{p-\tilde{\ell}(i),0\}=\sum_{i = 1}^k\min\{p-\hat{\ell}(i),0\}$ for every $\hat p\le p\le n$, $\sum_{i = 1}^k\min\{p-\tilde{\ell}(i),0\}\ge \sum_{i = 1}^k\min\{p-\hat{\ell}(i),0\}-1$ for every $1\le p\le \hat p$, and \eqref{newcondellhat}, we have $\tilde{\gamma}+\tilde\beta+\sum_{i = 1}^k\min\{p-\tilde{\ell}(i),0\}  \ge hp-e(p)$ for every $1\le p\le n.$ Finally, by $\tilde{\gamma}+\tilde{\ell}'(\mathbb Z_k)=\hat{\gamma}+\hat{\ell}'(\mathbb Z_k)$, $\sum_{i = 1}^k\min\{p-\tilde{\ell}'(i),0\}=\sum_{i = 1}^k\min\{p-\hat{\ell}'(i),0\}$ for every $\hat p\le p\le n$, $\sum_{i = 1}^k\min\{p-\tilde{\ell}'(i),0\}\ge \sum_{i = 1}^k\min\{p-\hat{\ell}'(i),0\}-1$ for every $1\le p\le \hat p$, and \eqref{newcdell'hat}, we have $\tilde{\gamma}+\tilde{\ell}'(\mathbb Z_k)+\sum_{i = 1}^k\min\{p-\tilde{\ell}'(i),0\}  \ge hp-e(p)$ for every $1\le p\le n.$

The existence of $\tilde\alpha,\tilde\beta,\tilde\gamma, \tilde\ell,$ and $\tilde\ell'$ contradicts the minimality of $\hat\gamma.$
\medskip

\noindent {\bf Case II.} Suppose that $X'=\mathbb Z_k.$ By Claim \ref{bonpop}(c),  \eqref{newcdell'hat}    is strict for every $1\le p\le n$. Then, \eqref{newcondellhat} is tight for $\hat p$, so, by Claim \ref{bonpop}(b), we have $X\neq \mathbb Z_k.$ Let $m\in \overline{X}.$ Then $\hat\ell'(m)\ge \hat p>\hat\ell(m).$ 
\medskip

Let $\tilde\gamma=\hat\gamma-1,$ $\tilde\alpha=\hat\alpha+1, \tilde\beta=\hat\beta+1, \tilde\ell(m)=\hat\ell(m)+1$, $\tilde\ell(i)=\hat\ell(i)$  for $i\in \mathbb Z_k-\{m\}$, and $\tilde\ell'(i)=\hat\ell'(i)$ for $i\in \mathbb Z_k.$
\medskip

By \eqref{newindcond} and $\hat\ell'(m)>\hat\ell(m),$ we have $n\ge\tilde\ell'(i)\ge\tilde\ell(i)$ for every $1\le i\le k.$ By \eqref{mvlkl}, we have $0\le\tilde{\ell}'(i)-\ell'(i)\le\tilde{\ell}(i)-\ell(i)$ for every $1\le i\le k.$ By Claim \ref{bonpop}(a), we have $hn\ge \tilde\alpha.$
By \eqref{newtotcond} and Claim \ref{bonpop}(c), we have $\tilde{\ell}'(\mathbb Z_k)\ge\tilde\beta\ge\tilde\alpha\ge\tilde{\ell} (\mathbb Z_k).$ By \eqref{uhviuycyihat}, we have $\tilde\alpha-\alpha=\tilde\beta-\beta=\tilde{\ell}(\mathbb Z_k)-\ell(\mathbb Z_k)\le{\gamma}-\tilde\gamma.$ By $\tilde{\gamma}+\tilde\beta=\hat{\gamma}+\hat\beta$, $\sum_{i = 1}^k\min\{p-\tilde{\ell}(i),0\}=\sum_{i = 1}^k\min\{p-\hat{\ell}(i),0\}$ for every $\hat p\le p\le n$, $\sum_{i = 1}^k\min\{p-\tilde{\ell}(i),0\}\ge \sum_{i = 1}^k\min\{p-\hat{\ell}(i),0\}-1$ for every $1\le p\le \hat p$, and \eqref{newcondellhat}, we have $\tilde{\gamma}+\tilde\beta+\sum_{i = 1}^k\min\{p-\tilde{\ell}(i),0\}  \ge hp-e(p)$ for every $1\le p\le n.$ Finally, by $\tilde{\gamma}+\tilde{\ell}'(\mathbb Z_k)=\hat{\gamma}+\hat{\ell}'(\mathbb Z_k)$, $\sum_{i = 1}^k\min\{p-\tilde{\ell}'(i),0\}\ge \sum_{i = 1}^k\min\{p-\hat{\ell}'(i),0\}-1$ for every $1\le p\le n$, and \eqref{newcdell'hat}    is strict for every $1\le p\le n$, we have $\tilde{\gamma}+\tilde{\ell}'(\mathbb Z_k)+\sum_{i = 1}^k\min\{p-\tilde{\ell}'(i),0\}  \ge hp-e(p)$ for every $1\le p\le n.$

The existence of $\tilde\alpha,\tilde\beta,\tilde\gamma, \tilde\ell,$ and $\tilde\ell'$ contradicts the minimality of $\hat\gamma.$
\medskip

 This completes the proof of Lemma \ref{efegegrg}.
 \end{proof}

We now are in a position to prove Theorem \ref{kjvkch}. Let $n=|V|.$ By \eqref{indivnecessary}, \eqref{hValpha}, and \eqref{totnecessary}, we have that \eqref{indnecess}, \eqref{miuviclkb} and \eqref{totnecess} hold.  Let $e(p)=\min\{e_\mathcal{A}(\mathcal{P}):\mathcal{P}$ subpartition of $V$ such that $|\mathcal{P}|=p\}$ for every $1\le p\le n.$ Note that $e(p)\ge 0$ for every $1\le p\le n.$ Then, by \eqref{kgeqh}, \eqref{konfeiobfiueguye1mregdirgam}, and \eqref{konfeiobfiueguye2mregdirgam}, we obtain \eqref{jkbfytdydyi}, \eqref{newcondellrgfr} and \eqref{newcondell'rrfgg}. Hence, by Lemma \ref{efegegrg}, there exist $\hat\alpha, \hat\beta\in\mathbb Z_+$ and $\hat\ell,\hat\ell': \mathbb Z_k\rightarrow  \mathbb Z_+$ such that \eqref{newindcond}--\eqref{newcondell'} hold. Then  \eqref{totnecessary}, \eqref{indivnecessary} \eqref{hValpha},   \eqref{konfeiobfiueguye1mregdir}, and \eqref{konfeiobfiueguye2mregdir} hold for $\hat\alpha, \hat\beta,\hat\ell$ and $\hat\ell'$. Hence, by Theorem \ref{jjjnbbbnajkhypregdir}, there exists an $h$-regular $(\hat\ell, \hat\ell')$-bordered $(\hat\alpha, \hat\beta)$-limited packing of hyperbranchings $\hat{\mathcal{B}}_i$ with root sets $\hat S_1,\dots, \hat S_k$  in $\mathcal{D}$. Let $i\in \mathbb{Z}_k.$ By \eqref{mvlkl}, $\hat\ell(i)\le|\hat S_i|\le\hat\ell'(i)$, and $\ell\ge 0$, there exists $S_i\subseteq \hat S_i$ with $|S_i|=|\hat S_i|-\hat\ell(i)+\ell(i).$ Further, $\ell(i)\le |S_i|\le\ell'(i).$ Let $\mathcal{B}_i$ be obtained from $\hat{\mathcal{B}}_i$ by adding a set $A_i$ of new arcs consisting of an arc from any vertex $s_i\in S_i$ to every vertex in $\hat S_i-S_i.$ Then $\mathcal{B}_i$ is an $S_i$-hyperbranching. Let $\mathcal{D}'$ be obtained from $\mathcal{D}$ by adding $\bigcup_{i=1}^kA_i.$ Since $\sum_{i=1}^k(|\hat S_i|-|S_i|)=\sum_{i=1}^k(\hat\ell(i)-\ell(i))$ and $\hat\alpha\le\sum_{i=1}^k|\hat S_i|\le\hat\beta,$ \eqref{uhviuycyi} implies that $\alpha\le\sum_{i=1}^k|S_i|\le\beta$. By \eqref{uhviuycyi}, the number of new arcs is $|\bigcup_{i=1}^kA_i|=\sum_{i=1}^k(\hat\ell(i)-\ell(i))\le\gamma$. Hence,  $\mathcal{B}_1,\dots, \mathcal{B}_k$ is the desired  $h$-regular $(\ell, \ell')$-bordered $(\alpha, \beta)$-limited packing of hyperbranchings in $\mathcal{D}'$ that completes the proof of Theorem \ref{kjvkch}.
\end{proof}

If $\gamma=0,$ then Theorem \ref{kjvkch} reduces to Theorem \ref{jjjnbbbnajkhypregdir}.

\subsection{Augmentation to have a packing of rooted hyperforests}

The following undirected version of Theorem \ref{kjvkch} can be obtained similarly by applying Lemma \ref{efegegrg} and Theorem \ref{jjjnbbbnajkhypregdirecece}.

\begin{thm}\label{kjvkchnon}
Let $\mathcal{G} = (V, \mathcal{E})$ be a hypergraph, $h, k, \alpha, \beta, \gamma \in \mathbb Z_+$ and $\ell, \ell': \mathbb{Z}_k \rightarrow  \mathbb{Z}_+$ such that~\eqref{totnecessary} and~\eqref{indivnecessary} hold. We can add at most $\gamma$ edges to $\mathcal{G}$ to have an $h$-regular $(\ell, \ell')$-bordered $(\alpha, \beta)$-limited packing of $k$ rooted hyperforests 
if and only if     \eqref{hValpha} and \eqref{kgeqh} hold and 
\begin{eqnarray*}  
	\gamma+\beta - \ell(\mathbb{Z}_k) + \sum_{i=1}^k\min\{|\mathcal{P}|,\ell(i)\} + e_\mathcal{E}(\mathcal{P}) &\ge& h|\mathcal{P}| \qquad \text{for every partition } \mathcal{P} \text{ of } V, \label{konfeiobfiueguye1mregdirgamnon}\\
	\gamma+\sum_{i=1}^k\min\{|\mathcal{P}|,\ell'(i)\} + e_\mathcal{E}(\mathcal{P}) &\ge& h|\mathcal{P}| \qquad \text{for every partition } \mathcal{P} \text{ of } V. \label{konfeiobfiueguye2mregdirgamnon}
\end{eqnarray*}
\end{thm}

If $\gamma=0,$ then Theorem \ref{kjvkchnon} reduces to Theorem \ref{jjjnbbbnajkhypregdirecece}.

\end{document}